\documentclass{amsart}[12pt]
\usepackage{amsmath}
\usepackage{amssymb}
\usepackage{amsthm}
\usepackage{verbatim}
\usepackage[pdftex]{graphics}

\newtheorem{theorem}{Theorem}

\title[Hardy-Littlewood inequalities] { Hardy-Littlewood inequalities for  norms\\ of positive operators on sequence  spaces }
\author{Miguel Lacruz}
\address{Departamento de An\'alisis Matem\'atico,  Universidad de Sevilla, Avenida Reina Mercedes, 41012 Seville, Spain}
\email{lacruz@us.es}
\thanks{This research was partially supported by Junta de Andaluc{\'\i}a under Proyecto de Excelencia FQM-3737, and by Ministerio de Ciencia e Innovaci\'on under Proyecto de Investigaci\'on  MTM2009-08934.} 
\date{\today}                                           
\begin{document}
\renewcommand{\thefootnote}{}
\begin{abstract}
We consider estimates of Hardy and Littlewood for norms of operators on sequence spaces, and we apply a factorization result of Maurey to obtain  improved estimates and simplified proofs for the special case of a positive operator.
\end{abstract}

\maketitle

In 1934, Hardy and Littlewood \cite{HL}, using powerful but technically difficult methods, extended results of Littlewood \cite{littlewood} and Toeplitz  \cite{toeplitz} to give   lower bounds for norms of bilinear forms on sequence spaces. In 2001, Osikiewicz and Tonge \cite{OT}, exploiting a deep interpolation theorem, obtained relatively simple proofs for the inequalities of Hardy and Littlewood. The aim of this paper is to provide  improved estimates and simplified proofs for the special case of a matrix operator with non-negative entries (we will call such operators positive).

If \(1 \leq p < \infty,\) we write \(\ell_p\) for the  complex vector space of all complex sequences \(x=(x_k)\) such that
\[
\|x\|_p \colon = \left ( \sum_k |x_k|^p \right )^{1/p} < \infty.
\]
We also write \(c_0\) for the space of all null complex sequences  \(x=(x_k)\) with the norm
\[
\|x\|_\infty \colon = \sup_k |x_k|.
\]
It turns out that  \(\ell_p\) and \(c_0\)  are Banach spaces under the indicated norms.  If \(X\) and \(Y\) are Banach spaces then we  write \({\mathcal B}(X,Y)\) for the complex vector space of all bounded linear operators \(A:X \to Y.\) This is a Banach space under the operator norm
\[
\|A\| \colon = \sup \{ \|Ax\| \colon x \in X, \|x\| \leq 1\}.
\]
When \(1 \leq p,q \leq \infty,\) it is convenient to  write \(\|A\|_{p,q}\) for the norm of an operator \(A \in {\mathcal B}(\ell_p,\ell_q).\) Every operator \(A \in {\mathcal B}(\ell_p,\ell_q)\) has a matrix representation \(A=(a_{j,k})\) with respect to the usual  bases.  The dual of a Banach space \(X\) is denoted by \(X^\ast.\)  If \(1 \leq p < \infty,\) it is a  standard fact that \(c_0^\ast=\ell_1\) and that  \(\ell_p^\ast=\ell_{p^\ast},\) where the conjugate index \(p^\ast\) is given by the expression
\[
\frac{1}{p} + \frac{1}{p^\ast} =1.
\]

Now the theorems of Hardy and Littlewood can be stated as follows:

\begin{theorem}[Hardy and Littlewood, \cite{HL}] 
\label{hardy}
Let \(1 \leq q \leq 2 \leq p \leq \infty\) and set \(1/r=1/q-1/p.\) There is an absolute constant  \(M=M(p,q)>0\) such that for every  \(A \in {\mathcal B}(\ell_p,\ell_q)\) we have
\[
\left [ \sum_j \left ( \sum_k |a_{j,k}|^2 \right )^{r/2} \right ]^{1/r}  \leq  M \| A\|_{p,q}, \qquad {\rm and}
\]
\[
\left [ \sum_k \left ( \sum_j |a_{j,k}|^2 \right )^{r/2} \right ]^{1/r}  \leq M \| A\|_{p,q}.
\]
\end{theorem}
\begin{theorem}[Hardy and Littlewood, \cite{HL}] 
\label{little}
Let \(1 \leq q \leq p \leq \infty,\) and set \(1/r=1/q-1/p,\) \(1/s=1/(2r)+1/4.\) There is an absolute constant \(M=M(p,q)>0\) such that for every  \(A \in {\mathcal B}(\ell_p,\ell_q)\) we have
\begin{enumerate}
\item[(i)] If \(r \geq 2\) then 
\[
\left ( \sum_{j,k} |a_{j,k}|^r \right )^{1/r} \leq M \|A\|_{p,q}.
\]
\item[(ii)] If \(r \leq 2\) then
\[
\left ( \sum_{j,k} |a_{j,k}|^s \right )^{1/s} \leq M \|A\|_{p,q}.
\]
\end{enumerate}
\end{theorem}
\noindent
Now we prove Theorem \ref{hardy} and Theorem \ref{little} for the special case of an operator with non negative entries. The key for our approach is a factorization theorem of Maurey \cite{maurey} that can be stated as follows:
\begin{theorem}[Maurey, \cite{maurey}]
\label{maurey1}
Let \(0 < q \leq p \leq \infty\) with \(p \geq 1,\) and set \(1/r=1/q-1/p.\) If  \(A \in {\mathcal B}(\ell_p,\ell_q)\) has non negative entries then there is a factorization \(A=DB,\) where \(B \in {\mathcal B}(\ell_p,\ell_p)\) has non negative entries and \(D \in {\mathcal B}(\ell_p,\ell_q)\) is a diagonal operator with non negative entries, say \(D={\rm diag}(d),\) with \(d \in \ell_r.\) Moreover,
\[
\|A\|_{p,q}= \inf \|d\|_r \cdot \|B\|_{p,p},
\]
where the infimum is taken over all  possible factorizations.
\end{theorem}

\begin{proof}[Proof of Theorem \ref{hardy} for a positive operator] Let \(A=(a_{j,k})\) with \(a_{j,k} \geq 0\) for all \(j,k.\) Take a factorization \(A=DB,\) where \(B \in {\mathcal B}(\ell_p,\ell_p)\) and \(D={\rm diag}(d),\) with \(d \in \ell_r.\)  We have \(a_{j,k}=d_j b_{j,k},\) so that
\begin{eqnarray*}
\left [ \sum_j \left ( \sum_k a_{j,k}^2 \right )^{r/2} \right ]^{1/r}  & = & \left [ \sum_j \left ( \sum_k d_j^2 b_{j,k}^2 \right )^{r/2} \right ]^{1/r}\\  
& = &  \left [ \sum_j d_j^r \left (\sum_k b_{j,k}^2 \right )^{r/2} \right ]^{1/r}\\
& \leq & \left ( \sum_j d_j^r \right )^{1/r} \sup_j  \left (\sum_k b_{j,k}^2 \right )^{1/2}\\
& = & \|d \|_r  \cdot \sup_j \|B^\ast e_j \|_2\\
& = & \|d\|_r \cdot \|B^\ast\|_{1,2}\\
& = & \|d\|_r \cdot \|B\|_{2,\infty}\\
& \leq &  \|d\|_r \cdot \|B\|_{p,p},
\end{eqnarray*}
and  taking the infimum over all    possible factorizations, we get the first inequality with   constant \(M=1.\) The second inequality follows easily from the first one by using duality. Indeed, set \(1/p^\ast = 1-1/p\) and  \(1/q^\ast = 1-1/q,\) so that \(p^\ast \leq 2 \leq q^\ast.\)  Let  \(A \in {\mathcal B}(\ell_p,\ell_q),\)  say \(A=(a_{j,k})\) with \(a_{j,k} \geq 0\) for all \(j,k,\) and notice that the adjoint operator \(A^\ast  \in {\mathcal B}(\ell_{q^\ast},\ell_{p^\ast})\) has the matrix representation \(A^\ast =(a_{k,j}).\) Hence,
\[
\left [ \sum_k \left ( \sum_j a_{j,k}^2 \right )^{r/2} \right ]^{1/r} \leq \|A^\ast\|_{q^\ast, p^\ast} = \|A\|_{p,q},
\]
as we wanted.
\end{proof}

\begin{proof}[Proof of Theorem \ref{little} for a positive operator] Let \(A=(a_{j,k})\) with \(a_{j,k} \geq 0\) for all \(j,k.\) Take a factorization \(A=DB,\) where \(B \in {\mathcal B}(\ell_p,\ell_p)\) and \(D={\rm diag}(d),\) with \(d \in \ell_r.\)  Let \(1/r^\ast = 1-1/r.\) Since \(q \geq 1,\) we have \(r^\ast \leq p.\)  We have \(a_{j,k}=d_j b_{j,k},\) so that
\begin{eqnarray*}
\left ( \sum_j \sum_k a_{j,k}^r \right )^{1/r} & = & \sum_j \left ( d_j^r  \sum_k b_{j,k}^r \right )^{1/r}\\
& \leq & \left ( \sum_j d_j^r \right )^{1/r}   \sup_j \left ( \sum_k b_{j,k} ^r\right )^{1/r}\\
& = & \|d \|_r  \cdot \sup_j \|B^\ast e_j \|_r\\
& = & \|d\|_r \cdot \|B^\ast\|_{1,r}\\
& = & \|d\|_r \cdot \|B\|_{r^\ast,\infty}\\
& \leq & \|d\|_r \cdot \|B\|_{p,p}
\end{eqnarray*}
and  taking the infimum over all   possible factorizations, we get the  inequality (i)  with  the constant \(M=1.\)   Notice that we have shown that  inequality (i) holds regardless of whether \(r \geq 2\) or \(r \leq 2.\) If \(r \leq 2\) then \(r \leq s \leq 2,\) so that we have obtained an improvement on Theorem \ref{little}, namely, that the left hand side in  inequality (ii) can be replaced by the left hand side in  inequality (i).
\end{proof}

\end{document}